\documentclass{amsart}

\usepackage{amsmath,amssymb,amsthm}
\usepackage{tikz}

\numberwithin{equation}{section}

\newtheorem{theorem}{Theorem}[section]
\newtheorem{lemma}[theorem]{Lemma}

\theoremstyle{definition}
\newtheorem{definition}[theorem]{Definition}
\newtheorem{question}[theorem]{Question}
\newtheorem{problem}[theorem]{Problem}
\theoremstyle{remark}
\newtheorem{remark}[theorem]{Remark}

\newcommand{\chr}{\operatorname{chr}}
\newcommand{\restr}{\mathbin{\upharpoonright}}

\title[Uncountable Complementary Graphs]
{Hedetniemi's Conjecture for Uncountable Complementary Graphs}
\author{L. Soukup}
\address{HUN-REN Alfr{\'e}d R{\'e}nyi Institute of Mathematics,
Budapest, Hungary}
\email{soukup@renyi.hu}
\subjclass[2020]{Primary 05C15; Secondary 03E35, 05C63, 05C76}
\keywords{chromatic number, categorical graph product, graph complement,
Hedetniemi's conjecture, forcing, diamond principle, Countryman line}

\begin{document}

\begin{abstract}
We study the complementary case of Hedetniemi's problem: can a graph
$G$ and its complement both be uncountably chromatic while their
categorical product is countably chromatic?  We prove that every such
example must have exactly $\omega_1$ vertices.  Nevertheless, assuming
$\diamondsuit$, we construct a graph $G$ on $\omega_1$ such that
$\chr(G)=\chr(\overline G)=\omega_1$ and
$\chr(G\times\overline G)=\omega$.  We also give a c.c.c.\ forcing of
size $\omega_1$ that adds such a graph.  We ask whether ZFC alone proves
that such a graph exists.
\end{abstract}

\maketitle

\section{Introduction}

All graphs in this note are simple and undirected.  If
$G_j=(V_j,E_j)$ for $j<2$, their (categorical) product is
\begin{displaymath}
 G_0\times G_1=
 \left(V_0\times V_1,
 \left\{\{(a_0,a_1),(b_0,b_1)\}:
       \{a_0,b_0\}\in E_0\text{ and }\{a_1,b_1\}\in E_1
 \right\}\right).
\end{displaymath}
For a graph $G=(V,E)$, write
$\overline G=(V,[V]^2\setminus E)$.  The study of inequalities relating
a graph and its complement goes back to Nordhaus and Gaddum
\cite{NordhausGaddum1956}.

Pulling back a coloring along either coordinate projection gives
\begin{displaymath}
 \chr(G_0\times G_1)\leq
 \min\{\chr(G_0),\chr(G_1)\}.
\end{displaymath}
Hedetniemi conjectured in 1966 that equality holds for finite graphs
\cite{Hedetniemi1966}, but Shitov disproved the conjecture
\cite{Shitov2019}; see \cite{Sauer2001,Zhu2025} for surveys.  For
infinite graphs, equality can fail much more dramatically.  Hajnal
showed that for every infinite cardinal $\kappa$ there are graphs
$G_0,G_1$ such that
\begin{displaymath}
 \chr(G_0)=\chr(G_1)=\kappa^+
 \quad\text{and}\quad
 \chr(G_0\times G_1)=\kappa.
\end{displaymath}
Thus two $\omega_1$-chromatic graphs may have a countably chromatic
product \cite{Hajnal1985}.  Soukup obtained two $\omega_2$-chromatic
graphs with countably chromatic product in a model of GCH
\cite{Soukup1988}.  Rinot later showed that in the constructible
universe, for every infinite cardinal $\kappa$, there are graphs
$\langle G_i:i<\kappa^+\rangle$ on $\kappa^+$ such that
\begin{displaymath}
 \chr(G_i)=\kappa^+
 \quad\text{and}\quad
 \chr(G_i\times G_j)=\omega
 \qquad(i<j<\kappa^+).
\end{displaymath}
Their edge sets may moreover be chosen pairwise disjoint
\cite{Rinot2017}.  Consequently, for $i\ne j$, the complement of
$G_i\cup G_j$ contains every $G_k$ with $k\notin\{i,j\}$, and hence has
chromatic number $\kappa^+$.  This suggests looking at the opposite
extreme, where two edge-disjoint graphs cover every pair---that is,
where the factors are a graph and its complement.

The central question of this paper is how large the chromatic numbers
of $G$ and $\overline G$ can be while their product remains countably
chromatic.  In particular, can one have
\begin{equation}\label{eq:central}
 \chr(G),\chr(\overline G)>\omega
 \quad\text{and}\quad
 \chr(G\times\overline G)=\omega.                 \tag{$*$}
\end{equation}

The complementary requirement imposes a strong restriction.

\begin{theorem}\label{thm:size-bound}
If $G$ and $\overline G$ are uncountably chromatic and
$G\times\overline G$ is countably chromatic, then
$|V(G)|=\omega_1$.
\end{theorem}

Thus $\omega_1$ is the only possible cardinality of such an example.
Despite these obstructions, this remaining possibility can be
realized.

\begin{theorem}\label{thm:diamond}
Assume $\diamondsuit$.  There is a graph $G$ on $\omega_1$ such that
\begin{displaymath}
 \chr(G)=\chr(\overline G)=\omega_1
 \quad\text{and}\quad
 \chr(G\times\overline G)=\omega.
\end{displaymath}
\end{theorem}

The proof uses two suitably chosen minimal Countryman lines.

\begin{theorem}\label{thm:main}
There is a c.c.c.\ partial order $P$ of cardinality $\omega_1$ such
that every $P$-generic extension contains a graph $G$ on $\omega_1$
for which
\begin{displaymath}
 \chr(G)=\chr(\overline G)=\omega_1
 \quad\text{and}\quad
 \chr(G\times\overline G)=\omega.
\end{displaymath}
\end{theorem}

\begin{problem}\label{q:zfc}
Does ZFC prove that there is a graph $G$ on $\omega_1$
satisfying~\eqref{eq:central}?
\end{problem}

\section{The negative results}

The proof of Theorem~\ref{thm:size-bound} rests on the following lemma,
based on an observation of Hajnal \cite{Hajnal1985}.

\begin{lemma}\label{lem:bound}
There is no graph $G$ such that
\begin{displaymath}
 \chr(G),\chr(\overline G)\geq\omega_2
 \quad\text{and}\quad
 \chr(G\times\overline G)\leq\omega.
\end{displaymath}
\end{lemma}

Since no proof of Hajnal's observation is given there, we include the
full argument for completeness.

\begin{proof}[Proof of Lemma~\ref{lem:bound}]
Suppose otherwise.  Let
$c:V(G)\times V(G)\to\omega$ be a proper coloring of
$G\times\overline G$, and choose $X\in[V(G)]^{\omega_1}$.  We first
show that $G\restr X$ is countably chromatic.  Otherwise, for every
$y\in V(G)$ the map
\begin{displaymath}
 x\longmapsto c(x,y),\qquad x\in X,
\end{displaymath}
fails to be a proper coloring of $G\restr X$.  Choose an edge
$e_y=\{a_y,b_y\}\in E(G)\cap[X]^2$ and an $n_y<\omega$ such that
\begin{displaymath}
 c(a_y,y)=c(b_y,y)=n_y.
\end{displaymath}
Then
\begin{displaymath}
 y\longmapsto(e_y,n_y)
\end{displaymath}
is a proper coloring of $\overline G$: if $\{y,y'\}\in E(\overline G)$
and $(e_y,n_y)=(e_{y'},n_{y'})$, writing
$e_y=e_{y'}=\{a,b\}$ and $n_y=n_{y'}=n$ would make $(a,y)$ and
$(b,y')$ adjacent vertices of $G\times\overline G$ with the same
$c$-color.  Since
\begin{displaymath}
 |(E(G)\cap[X]^2)\times\omega|\leq\omega_1,
\end{displaymath}
this contradicts $\chr(\overline G)\geq\omega_2$.  Hence
$G\restr X$ is countably chromatic.  Interchanging the two coordinates
and replacing $G$ by $\overline G$ shows that
$\overline G\restr X$ is countably chromatic as well.  Let
$c_0,c_1:X\to\omega$ be proper colorings of these two graphs.  The map
\begin{displaymath}
 x\longmapsto\bigl(c_0(x),c_1(x)\bigr)
\end{displaymath}
is injective, because every pair of distinct elements of $X$ is an
edge of exactly one of the two complementary graphs.  This contradicts
$|X|=\omega_1$.
\end{proof}

\begin{proof}[Proof of Theorem~\ref{thm:size-bound}]
By Lemma~\ref{lem:bound}, after interchanging $G$ and $\overline G$
if necessary, we have $\chr(G)=\omega_1$.  Suppose that
$|V(G)|>\omega_1$, and let $d:V(G)\to\omega_1$ be a proper coloring.
Some color class $C$ of $d$ is uncountable, since otherwise $V(G)$
would be the union of $\omega_1$ countable sets.  Thus $C$ is a clique
in $\overline G$.

Let $c:V(G)\times V(G)\to\omega$ be a proper coloring of
$G\times\overline G$.  For every $x\in V(G)$, choose $n_x<\omega$ such
that
\begin{displaymath}
 A_x=\{y\in C:c(x,y)=n_x\}
\end{displaymath}
is uncountable.  The map $x\mapsto n_x$ is a proper coloring of
$G$: if $\{x,x'\}\in E(G)$ and $n_x=n_{x'}$, choose distinct
$y\in A_x$ and $y'\in A_{x'}$.  Then $(x,y)$ and $(x',y')$ are
adjacent in $G\times\overline G$ and have the same $c$-color, a
contradiction.  Hence $G$ is countably chromatic, again a
contradiction.  Therefore $|V(G)|=\omega_1$.
\end{proof}

\section{A construction from $\diamondsuit$}

Recall that a linear order is \emph{scattered} if it contains no copy
of the rationals, and it is \emph{$\sigma$-scattered} if it is the
union of countably many scattered suborders.  A non-$\sigma$-scattered
linear order $L$ is \emph{minimal} if $L$ embeds into each of its
non-$\sigma$-scattered suborders.  A \emph{Countryman line} is a linear
order $L=(X,<_{L})$ of cardinality $\omega_1$ such that the
coordinatewise order on $X^2$ is the union of countably many chains.

We use the following consequence of a theorem of Shalev
\cite[Corollary~1.2 and Remark~1.3]{Shalev2024}: $\diamondsuit$ implies
that there is a family of $2^{\omega_1}$ pairwise non-embeddable
Countryman lines, each minimal with respect to being
non-$\sigma$-scattered.

\begin{proof}[Proof of Theorem~\ref{thm:diamond}]
Choose two members $L_0,L_1$ of Shalev's family such that
\begin{equation}\label{eq:nonembedding}
 L_0\not\hookrightarrow L_1
 \quad\text{and}\quad
 L_0\not\hookrightarrow L_1^*.
\end{equation}
Here $L^*$ denotes the reverse of $L$.  To see that this choice is
possible, fix $L_0$.  There is at most one member $M$ of the
family for which $L_0$ embeds into $M^*$.  Indeed, if both $M$ and
$N$ had this property, then $L_0^*$ would embed into both $M$ and
$N$.  Its images are non-$\sigma$-scattered, so the minimality of
$M$ and $N$ would give embeddings
\begin{displaymath}
 M\longrightarrow L_0^*\longrightarrow N,
\end{displaymath}
contrary to the pairwise non-embeddability of the family.  We may
therefore choose $L_1$ outside this possible exception; the first part
of~\eqref{eq:nonembedding} follows directly from pairwise
non-embeddability.

Relabel the two orders so that they have the same underlying set $X$
of cardinality $\omega_1$, and write them as $(X,<_{0})$ and
$(X,<_{1})$.  Define a graph $G$ on $X$ by declaring distinct $x,y$
adjacent precisely when the two orders agree on their orientation:
\begin{equation}\label{eq:edge-definition}
 \{x,y\}\in E(G)
 \quad\Longleftrightarrow\quad
 (x<_{0}y\Longleftrightarrow x<_{1}y).
\end{equation}

We first check the chromatic numbers of the two complementary graphs.
If $G$ had a countable proper coloring, write $X=\bigcup_{n<\omega}A_n$
with every $A_n$ independent.  Some $L_0\restr A_n$ is
non-$\sigma$-scattered, since a countable union of
$\sigma$-scattered orders is $\sigma$-scattered.  Minimality gives
an embedding $L_0\longrightarrow L_0\restr A_n$.  On $A_n$ the two
orders disagree on every pair, so the identity embeds
$L_0\restr A_n$ into $L_1^*$.  Their composition contradicts
\eqref{eq:nonembedding}.
Thus $\chr(G)=\omega_1$.  If $\overline G$ had a countable proper
coloring, the same argument would use a color class on which the two
orders agree and give an embedding $L_0\longrightarrow L_1$.
Consequently,
\begin{equation}\label{eq:factor-chromatic-numbers}
 \chr(G)=\chr(\overline G)=\omega_1.
\end{equation}

For $j<2$, choose a cover
\begin{displaymath}
 X^2=\bigcup_{n<\omega}C_n^j
\end{displaymath}
in which every $C_n^j$ is a chain in the coordinatewise order induced
by $<_{j}$.  Choose $k_j(x,y)<\omega$ with
$(x,y)\in C_{k_j(x,y)}^j$, and put
\begin{displaymath}
 f(x,y)=\bigl(k_0(x,y),k_1(x,y)\bigr)\in\omega^2.
\end{displaymath}
We claim that $f$ is a proper countable coloring of
$G\times\overline G$.  Suppose that
$f(x,y)=f(x',y')$, where $x\ne x'$ and $y\ne y'$.  For each
$j<2$, the points $(x,y)$ and $(x',y')$ lie in one coordinatewise
chain.  Hence
\begin{displaymath}
 x<_{j}x'\quad\Longleftrightarrow\quad y<_{j}y'.
\end{displaymath}
It follows that the two orders agree on $\{x,x'\}$ exactly when they agree on
$\{y,y'\}$.  By~\eqref{eq:edge-definition}, one of these pairs cannot
be an edge of $G$ and the other an edge of $\overline G$.  This proves
the claim and gives
$\chr(G\times\overline G)\leq\omega$.

For completeness, we show that the last inequality is an equality.
Let $W=W_0\mathbin{\dot\cup}W_1$ be countably infinite, disjoint from
$X$, with both pieces infinite.  Extend $G$ to a graph $\widehat G$ on
$X\cup W$ by adding all pairs from $[W_0]^2$ and no other pairs
involving $W$.  The restrictions to $X$ show that both $\widehat G$
and its complement still have chromatic number $\omega_1$.

Fix an injection $e:W\to\omega$.  Using three disjoint countable
palettes (whose union we identify with $\omega$), extend the product
coloring by
\begin{equation}\label{eq:extended-product-coloring}
 F(u,v)=
 \begin{cases}
  (0,f(u,v)),&u,v\in X,\\
  (1,e(u)),&u\in W,\\
  (2,e(v)),&u\in X, v\in W.
 \end{cases}
\end{equation}
Two vertices receiving the same color of the second kind have the
same first coordinate, and two receiving the same color of the third
kind have the same second coordinate; neither pair can be adjacent in
the categorical product.  Thus~\eqref{eq:extended-product-coloring}
properly colors
$\widehat G\times\overline{\widehat G}$ with countably many colors.
On the other hand, pairing the members of the clique $W_0$ bijectively
with the members of the $\widehat G$-independent set $W_1$ produces a
copy of $K_\omega$ in the product.  Therefore
\begin{displaymath}
 \chr(\widehat G\times\overline{\widehat G})=\omega.
\end{displaymath}
Finally, relabel $X\cup W$ as $\omega_1$.
\end{proof}

The preceding argument cannot be carried out in ZFC as stated: it
requires suitably chosen minimal Countryman lines, and ZFC does not
even prove that a minimal Countryman line exists.  Indeed, Moore proved
that it is consistent with CH that $\omega_1$ and $\omega_1^*$ are the
only minimal uncountable linear orders \cite{Moore2007}; see also
\cite[Theorem~1.10]{CummingsEisworthMoore2024}.  Since every
uncountable suborder of a Countryman line is again a Countryman line,
that model has no minimal Countryman lines.  The forcing construction
below shows that the graph-theoretic phenomenon is nevertheless not
intrinsically tied to $\diamondsuit$.  We define a c.c.c.\ forcing of
cardinality $\omega_1$ that adds the required graph over any ground
model.  In particular, forcing over a model of $\neg$CH shows that the
existence of such a graph is consistent with $\neg$CH.  The argument
also isolates the finite extension and amalgamation properties
underlying the construction.

\section{The forcing construction}

\begin{definition}\label{def:preconditions}
A \emph{precondition} is a triple
\begin{displaymath}
 p=(A^p,h^p,f^p),
\end{displaymath}
where $A^p\in[\omega_1]^{<\omega}$,
$h^p:[A^p]^2\to 2$, and $f^p:A^p\times A^p\to\omega$.
We order preconditions by coordinatewise extension:
\begin{displaymath}
 p\leq q
 \quad\Longleftrightarrow\quad
 A^p\supseteq A^q,\quad h^p\supseteq h^q,\quad
 f^p\supseteq f^q.
\end{displaymath}
Put
\begin{displaymath}
 E_i^p=(h^p)^{-1}\{i\},\qquad H_i^p=(A^p,E_i^p)
 \quad(i<2).
\end{displaymath}
Thus $H_0^p$ and $H_1^p$ are complementary graphs on $A^p$.
For $e<2$ define
\begin{displaymath}
 f_e^p(a,b)=
 \begin{cases}
 f^p(a,b),&e=0,\\
 f^p(b,a),&e=1.
 \end{cases}
\end{displaymath}
\smallskip
\noindent\emph{Synchronized pairs.}
Fix $x\in A^p$ and $e<2$.  Let
$\mathsf{Sync}^{x,e}_p\subseteq A^p\times A^p$ consist of all pairs
$(a,b)$ for which there are $k<\omega$ and two sequences
\begin{displaymath}
 a_0,\ldots,a_k\quad\text{and}\quad b_0,\ldots,b_k
\end{displaymath}
such that
\begin{enumerate}
\item $a_0=b_0=x$, $a_k=a$, and $b_k=b$;
\item $a_j,b_j\geq x$ for every $j\leq k$;
\item
$f_e^p(a_j,a_{j+1})=f_e^p(b_j,b_{j+1})$ for every $j<k$.
\end{enumerate}
We call such a pair of sequences \emph{synchronized walks} witnessing
$(a,b)\in\mathsf{Sync}^{x,e}_p$.
The case $k=0$ is allowed.  Notice also that, whenever $x\leq a$,
the identical one-step walks show
$(a,a)\in\mathsf{Sync}^{x,e}_p$.

\begin{figure}[ht]
\centering
\begin{tikzpicture}[
  vertex/.style={inner sep=1pt},
  every node/.style={font=\small}
]
\node[vertex] (x) at (0,0) {$x=a_0=b_0$};
\node (aldots) at (1.4,0.9) {$\cdots$};
\node[vertex] (aj) at (2.8,0.9) {$a_j$};
\node[vertex] (ajp) at (4.6,0.9) {$a_{j+1}$};
\node (ardots) at (5.8,0.9) {$\cdots$};
\node[vertex] (ak) at (7.2,0.9) {$a_k=a$};

\node (bldots) at (1.4,-0.9) {$\cdots$};
\node[vertex] (bj) at (2.8,-0.9) {$b_j$};
\node[vertex] (bjp) at (4.6,-0.9) {$b_{j+1}$};
\node (brdots) at (5.8,-0.9) {$\cdots$};
\node[vertex] (bk) at (7.2,-0.9) {$b_k=b$};

\draw[->] (x) -- (aldots);
\draw[->] (aldots) -- (aj);
\draw[->] (aj) -- node[above] {$c_j$} (ajp);
\draw[->] (ajp) -- (ardots);
\draw[->] (ardots) -- (ak);
\draw[->] (x) -- (bldots);
\draw[->] (bldots) -- (bj);
\draw[->] (bj) -- node[below] {$c_j$} (bjp);
\draw[->] (bjp) -- (brdots);
\draw[->] (brdots) -- (bk);

\draw[densely dotted] (aj) -- (bj);
\draw[densely dotted] (ajp) -- (bjp);
\draw[densely dotted] (ak) -- node[midway,right]
  {$(a,b)\in\mathsf{Sync}^{x,e}_p$} (bk);

\node[align=center] at (3.6,-1.9) {%
  $c_j=f_e^p(a_j,a_{j+1})=f_e^p(b_j,b_{j+1})\quad(j<k),$\\[2pt]
  $a_i,b_i\geq x\quad(i\leq k).$};
\end{tikzpicture}
\caption{Synchronized walks witnessing
$(a,b)\in\mathsf{Sync}^{x,e}_p$.  Corresponding steps have the same
$f_e^p$-color, and all vertices of the walks are at least $x$.}
\label{fig:sync}
\end{figure}
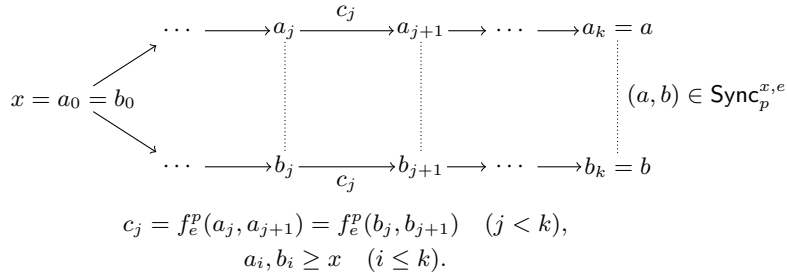
\end{definition}

\begin{remark}\label{rem:recursive-sync}
Equivalently, $\mathsf{Sync}^{x,e}_p$ is the least relation
$R\subseteq A^p\times A^p$ satisfying the following two rules:
\begin{enumerate}
\item[(R0)] $(x,x)\in R$;
\item[(R1)] if $(a,b)\in R$ and $a',b'\in A^p$ satisfy
\begin{displaymath}
 a',b'\geq x
 \quad\text{and}\quad
 f_e^p(a,a')=f_e^p(b,b'),
\end{displaymath}
then $(a',b')\in R$.
\end{enumerate}
Indeed, iterating (R1) produces exactly the synchronized walks in the
definition above.
\end{remark}

\begin{definition}[Good preconditions]\label{def:good}
A precondition $p$ is \emph{good} if the following two requirements
hold.
\begin{enumerate}
\item[(G1)] $f^p$ is a proper coloring of $H_0^p\times H_1^p$.
Equivalently, whenever $a_0\ne b_0$, $a_1\ne b_1$,
\begin{displaymath}
 h^p(\{a_0,b_0\})=0
 \quad\text{and}\quad
 h^p(\{a_1,b_1\})=1,
\end{displaymath}
then
\begin{displaymath}
 f^p(a_0,a_1)\ne f^p(b_0,b_1).
\end{displaymath}
\item[(G2)] If $e<2$, $x,a,b,u,v\in A^p$, $u\ne v$,
\begin{displaymath}
 (a,b)\in\mathsf{Sync}^{x,e}_p,\qquad
 \min\{u,v\}<x,
 \qquad h^p(\{u,v\})=1-e,
\end{displaymath}
then
\begin{displaymath}
 f_e^p(a,u)\ne f_e^p(b,v).
\end{displaymath}
\end{enumerate}

\begin{figure}[ht]
\centering
\begin{tikzpicture}[
  vertex/.style={inner sep=1pt},
  every node/.style={font=\small}
]
\node[vertex] (x) at (0,0) {$x$};
\node[vertex] (a1) at (1.6,0.9) {$a_1$};
\node (adots) at (2.8,0.9) {$\cdots$};
\node[vertex] (a) at (4.2,0.9) {$a$};
\node[vertex] (u) at (7.0,0.9) {$u$};
\node[vertex] (b1) at (1.6,-0.9) {$b_1$};
\node (bdots) at (2.8,-0.9) {$\cdots$};
\node[vertex] (b) at (4.2,-0.9) {$b$};
\node[vertex] (v) at (7.0,-0.9) {$v$};

\draw[->] (x) -- node[above,sloped] {$c_0$} (a1);
\draw[->] (a1) -- (adots);
\draw[->] (adots) -- node[above] {$c_{k-1}$} (a);
\draw[->] (x) -- node[below,sloped] {$c_0$} (b1);
\draw[->] (b1) -- (bdots);
\draw[->] (bdots) -- node[below] {$c_{k-1}$} (b);
\draw[->] (a) -- node[above]
  {$d_0:=f_e^p(a,u)$} (u);
\draw[->] (b) -- node[below]
  {$d_1:=f_e^p(b,v)$} (v);

\draw[densely dotted] (a1) -- (b1);
\node[align=center] at (4.2,0)
  {$(a,b)\in\mathsf{Sync}^{x,e}_p$};
\draw[thick] (u) -- node[midway,right,align=left]
  {$u\ne v,$\\$h^p(\{u,v\})=1-e$} (v);

\node[align=center] at (4.2,-1.9)
  {$\min\{u,v\}<x\quad\Longrightarrow\quad d_0\ne d_1\qquad\text{(G2)}$};
\end{tikzpicture}
\caption{Requirement (G2).  If
$(a,b)\in\mathsf{Sync}^{x,e}_p$, $\min\{u,v\}<x$, and
$h^p(\{u,v\})=1-e$, then the terminal colors
$d_0=f_e^p(a,u)$ and $d_1=f_e^p(b,v)$ are distinct.}
\label{fig:g2}
\end{figure}
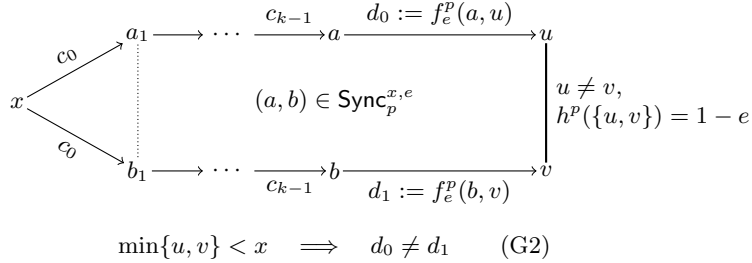
\end{definition}

Let $P$ be the poset of all good preconditions, ordered by coordinatewise
extension.  The next two lemmas establish the extension and amalgamation
properties needed for $P$.

\begin{lemma}[Fresh extensions]\label{lem:fresh}
Let $p$ be good, let $A^p\subseteq A\in[\omega_1]^{<\omega}$, and
let $h:[A]^2\to2$ extend $h^p$.  Then there is a map
$f:A^2\to\omega$ such that $(A,h,f)$ is a good extension of $p$.
\end{lemma}

\begin{proof}
Extend $f^p$ to $f:A^2\to\omega$ by giving the members of
$A^2\setminus(A^p)^2$ pairwise distinct values outside
$\operatorname{ran}(f^p)$, and set $q=(A,h,f)$.  Call the ordered
pairs in $(A^p)^2$ \emph{old} and those in
$A^2\setminus(A^p)^2$ \emph{new}.  The choice of $f$ gives the
following \emph{freshness property}:
\begin{displaymath}
 s\in A^2,\quad t\in A^2\setminus(A^p)^2,
 \quad\text{and}\quad f(s)=f(t)
 \quad\Longrightarrow\quad s=t.                       \tag{F}
\end{displaymath}
To verify (G1), suppose that two
vertices of $H_0^q\times H_1^q$ have the same $f$-color.  If either of
the corresponding ordered pairs is new, (F) says
that the two ordered pairs are identical, so they cannot be adjacent.
If both pairs belong to $(A^p)^2$, use (G1) for $p$.

We verify (G2).  By symmetry between the two coordinates---interchanging
$H_0$ and $H_1$ and transposing $f$---it is enough to consider $e=0$.
Fix a witnessing pair of synchronized walks for
$(a,b)\in\mathsf{Sync}^{x,0}_q$, and suppose, toward a contradiction,
that
\begin{equation}\label{eq:fresh-g2-counterexample}
 h(\{u,v\})=1,\quad \min\{u,v\}<x,\quad
 f(a,u)=f(b,v).
\end{equation}
Property (F) and $u\ne v$ imply that the ordered pairs $(a,u)$ and
$(b,v)$ both belong to $(A^p)^2$.  In particular
$a,b,u,v\in A^p$.

Let $a_0,\ldots,a_k$ and $b_0,\ldots,b_k$ be the two synchronized
walks.  By synchronization,
\begin{displaymath}
 f(a_j,a_{j+1})=f(b_j,b_{j+1})\qquad(j<k).
\end{displaymath}
If either $(a_j,a_{j+1})$ or $(b_j,b_{j+1})$ is new, (F) implies
$(a_j,a_{j+1})=(b_j,b_{j+1})$.  If $a=b$, then
$(a,a)\in\mathsf{Sync}^{a,0}_p$ and
$\min\{u,v\}<x\leq a$.  Thus (G2) for $p$, applied with base point
$a$, gives $f(a,u)\ne f(a,v)$, whereas
\eqref{eq:fresh-g2-counterexample}, since $a=b$, gives
$f(a,u)=f(a,v)$, a contradiction.
Assume that $a\ne b$, and let $j$ be the last index at which the two
walks agree.  Every step from index $j$ onward is old, and all vertices
in this terminal part of the walks belong to $A^p$.  Let $z$ be the
least ordinal occurring in this terminal part.  If $z=a_j=b_j$, use
the two terminal sequences.  Otherwise use
\begin{displaymath}
 z,a_j,a_{j+1},\ldots,a_k=a
 \quad\text{and}\quad
 z,b_j,b_{j+1},\ldots,b_k=b.
\end{displaymath}
The first steps of these two sequences are identical, and the remaining
corresponding steps have the same $f^p$-colors.  Thus in either case
\begin{displaymath}
 (a,b)\in\mathsf{Sync}^{z,0}_p.
\end{displaymath}
All vertices of the terminal part are at least $z$, while
$\min\{u,v\}<x\leq z$.  Now (G2) for $p$ again contradicts
\eqref{eq:fresh-g2-counterexample}.
\end{proof}

We next isolate the amalgamation used both for the chain condition and
for the chromatic-number argument.

\begin{definition}[Twins and separated twins]\label{def:twins}
Two good preconditions $p$ and $q$
are called \emph{twins} if $|A^p|=|A^q|$ and the increasing bijection
$\sigma:A^p\to A^q$ fixes $R=A^p\cap A^q$ pointwise and is an
isomorphism, that is,
\begin{displaymath}
 h^q(\{\sigma(a),\sigma(b)\})=h^p(\{a,b\}),
 \qquad
 f^q(\sigma(a),\sigma(b))=f^p(a,b).
\end{displaymath}
Write $D_p=A^p\setminus R$ and $D_q=A^q\setminus R$.
The twins are \emph{separated} if
\begin{displaymath}
 R<D_p<D_q,
\end{displaymath}
where $X<Y$ means that every member of $X$ is smaller than every
member of $Y$.
\end{definition}

\begin{lemma}[Amalgamation lemma]\label{lem:amalgamation}
Suppose that $p$ and $q$ are separated twins and that $D_p\ne\varnothing$.
Let $x=\min D_p$.  For each $i<2$ there is a good precondition
$r\leq p,q$ such that
\begin{displaymath}
 h^r(\{x,\sigma(x)\})=i.
\end{displaymath}
\end{lemma}

\begin{proof}
We prove the case $i=0$; the case $i=1$ is analogous.

Put $A^r=A^p\cup A^q$, and define the projection $\pi:A^r\to A^p$ by
\begin{displaymath}
 \pi\restr A^p=\operatorname{id}_{A^p},
 \qquad \pi\restr D_q=\sigma^{-1}.
\end{displaymath}
Let
\begin{equation}\label{eq:special-cross-edges}
 \mathcal F=
 \bigl\{\{a,b\}:a\in D_p, b\in D_q,
        (a,\pi(b))\in\mathsf{Sync}^{x,0}_p\bigr\}.
\end{equation}
On $[A^p]^2\cup[A^q]^2$ let $h^r$ agree with $h^p\cup h^q$.
For $a\in D_p$ and $b\in D_q$, put
\begin{equation}\label{eq:amalgam-h}
 h^r(\{a,b\})=
 \begin{cases}
 0,&\{a,b\}\in\mathcal F,\\
 h^p(\{a,\pi(b)\}),&\{a,b\}\notin\mathcal F.
 \end{cases}
\end{equation}
Formula~\eqref{eq:amalgam-h} is well-defined.  Indeed, if $a=\pi(b)$,
the identical
walk from $x$ to $a$ puts $\{a,b\}$ in $\mathcal F$, so the first line
applies.  Otherwise $\{a,\pi(b)\}\in[A^p]^2$, and the second line is
defined.

Let $f^r$ extend $f^p\cup f^q$, and give all ordered pairs in
\begin{equation}\label{eq:new-ordered-pairs}
 (A^r)^2\setminus\bigl((A^p)^2\cup(A^q)^2\bigr)
\end{equation}
pairwise distinct values outside $\operatorname{ran}(f^p)$.
The functions $f^p$ and $f^q$ agree on their common domain, so $f^r$ is
well-defined.  Call the ordered pairs in
\eqref{eq:new-ordered-pairs} \emph{new} and all other
ordered pairs \emph{old}.  Since $p$ and $q$ are twins,
$\operatorname{ran}(f^q)=\operatorname{ran}(f^p)$; hence $f^r$ has the
analogous freshness property:
\begin{displaymath}
 \begin{aligned}
 s&\in(A^r)^2,\\
 t&\in(A^r)^2\setminus\bigl((A^p)^2\cup(A^q)^2\bigr),
 \quad\text{and}\quad f^r(s)=f^r(t)
 \quad\Longrightarrow\quad s=t.
 \end{aligned}                                           \tag{$F_r$}
\end{displaymath}
Set $r=(A^r,h^r,f^r)$.  Since
$(x,x)\in\mathsf{Sync}^{x,0}_p$,
\eqref{eq:special-cross-edges}--\eqref{eq:amalgam-h} give
$h^r(\{x,\sigma(x)\})=0$.

We first prove (G1).  Suppose otherwise.  Thus, for suitable
$a_0,a_1\in A^r$ and $b_0,b_1\in A^r$,
\begin{equation}\label{eq:g1-counterexample}
 h^r(\{a_0,b_0\})=0,\qquad
 h^r(\{a_1,b_1\})=1,\qquad
 f^r(a_0,a_1)=f^r(b_0,b_1).
\end{equation}
If either ordered pair in \eqref{eq:g1-counterexample} is new, $(F_r)$
forces the two ordered pairs to be identical, which is impossible.
If both ordered pairs belong to $(A^p)^2$, or both belong to
$(A^q)^2$, use (G1) in that condition.  We may therefore assume,
after interchanging the two product vertices if necessary, that
$(a_0,a_1)\in(A^p)^2$ and $(b_0,b_1)\in(A^q)^2$.
Put $b'_j=\pi(b_j)$ for $j<2$.  Then
\begin{equation}\label{eq:projected-color-equality}
 f^p(a_0,a_1)=f^p(b'_0,b'_1).
\end{equation}

For $j<2$, call the edge $\{a_j,b_j\}$ \emph{special} if
$\{a_j,b_j\}\in\mathcal F$.  A nonspecial edge has the same color as the genuine
edge $\{a_j,b'_j\}$ of $p$; its projected endpoints are distinct,
since equality would make the edge special by the diagonal observation
above.  If neither
edge is special,
\eqref{eq:g1-counterexample}--\eqref{eq:projected-color-equality}
contradict (G1) for $p$.  Every special edge has color $0$, so the only
remaining possibility is that the $0$-th edge is special.  Hence
\begin{equation}\label{eq:special-edge-data}
 (a_0,b'_0)\in\mathsf{Sync}^{x,0}_p,\qquad
 h^p(\{a_1,b'_1\})=1.
\end{equation}
Let
\begin{displaymath}
 c_0,\ldots,c_k=a_0\quad\text{and}\quad
 d_0,\ldots,d_k=b'_0
\end{displaymath}
be synchronized walks from $x$ witnessing the first assertion in
\eqref{eq:special-edge-data}.  If $a_1,b'_1\geq x$, extend these walks
by appending $a_1$ to the first and $b'_1$ to the second.  By
\eqref{eq:projected-color-equality}, the two added steps have the same
$f^p$-color.  The extended walks therefore witness
$(a_1,b'_1)\in\mathsf{Sync}^{x,0}_p$.  Since $b'_1=\pi(b_1)$,
\eqref{eq:special-cross-edges} gives $\{a_1,b_1\}\in\mathcal F$; hence
the other edge is also special, a contradiction.  Otherwise
$\min\{a_1,b'_1\}<x$, and (G2) for $p$, applied to
\eqref{eq:special-edge-data}, contradicts
\eqref{eq:projected-color-equality}.  Thus (G1) holds in $r$.

It remains to prove (G2).  Recall that $x=\min D_p$ is fixed throughout
the construction; we use $y$ for the arbitrary base point occurring
in (G2).  It is enough to consider $e=0$; the case $e=1$ is analogous.
Suppose that
\begin{equation}\label{eq:g2-counterexample-data}
 (a,b)\in\mathsf{Sync}^{y,0}_r,\qquad h^r(\{u,v\})=1,\qquad
 \min\{u,v\}<y,
\end{equation}
but
\begin{equation}\label{eq:g2-counterexample-color}
 f^r(a,u)=f^r(b,v).
\end{equation}
By $(F_r)$ and $u\ne v$, the ordered pairs $(a,u)$ and $(b,v)$ are
both old: otherwise \eqref{eq:g2-counterexample-color} would make them
identical and hence give $u=v$.
Let $a_0,\ldots,a_k=a$ and $b_0,\ldots,b_k=b$ be synchronized walks
from $y$ witnessing the first assertion in
\eqref{eq:g2-counterexample-data}.  For every $j<k$,
\begin{displaymath}
 f^r(a_j,a_{j+1})=f^r(b_j,b_{j+1}).
\end{displaymath}
If either $(a_j,a_{j+1})$ or $(b_j,b_{j+1})$ is new, $(F_r)$ implies
that these two ordered pairs are identical.

There are three cases.  If $y\in R$, project both witnessing walks by
$\pi$.  Fix $j<k$.  If both $(a_j,a_{j+1})$ and
$(b_j,b_{j+1})$ are old, then the definition of $\pi$ and the fact that
$p$ and $q$ are twins give
\begin{displaymath}
 f^p(\pi(a_j),\pi(a_{j+1}))
 =f^p(\pi(b_j),\pi(b_{j+1})).
\end{displaymath}
If either ordered pair is new, the preceding paragraph shows that the
two ordered pairs, and hence their projections, are identical.  Thus
the displayed equality holds for every $j<k$.  All projected vertices
are at least $y$, so
\begin{displaymath}
 (\pi(a),\pi(b))\in\mathsf{Sync}^{y,0}_p.
\end{displaymath}
The edge $\{u,v\}$ cannot join $D_p$ to $D_q$, since both sets lie
above every member of $R$, whereas $\min\{u,v\}<y\in R$.  Consequently
\begin{displaymath}
 h^p(\{\pi(u),\pi(v)\})=1,
 \qquad \min\{\pi(u),\pi(v)\}<y.
\end{displaymath}
Moreover, projecting \eqref{eq:g2-counterexample-color} gives
\begin{displaymath}
 f^p(\pi(a),\pi(u))=f^p(\pi(b),\pi(v)).
\end{displaymath}
Together with
$(\pi(a),\pi(b))\in\mathsf{Sync}^{y,0}_p$, these statements contradict
(G2) for $p$.

If $y\in D_q$, every vertex of the witnessing walks lies in $D_q$.
Thus the first assertion in \eqref{eq:g2-counterexample-data} already
holds in $q$.  Oldness of the two pairs in
\eqref{eq:g2-counterexample-color} implies $u,v\in A^q$, and (G2) for
$q$ gives the
contradiction.

Finally suppose that $y\in D_p$.  An induction on $j\leq k$ shows that
either $a_j=b_j$, or $a_j,b_j\in D_p$, or $a_j,b_j\in D_q$.  Indeed,
all vertices of the walks lie above $R$ and hence belong to
$D_p\cup D_q$.  If $(a_j,a_{j+1})$ is old, then $a_j$ and $a_{j+1}$
belong to the same one of $D_p,D_q$; the analogous statement holds for
$(b_j,b_{j+1})$.  If either ordered pair is new, $(F_r)$ makes the two
ordered pairs identical.  It follows from
\eqref{eq:g2-counterexample-color} that
$a,b,u,v$ all lie in the same
support, either $A^p$ or $A^q$.  Let $s\in\{p,q\}$ be the corresponding
condition.  Thus $\{u,v\}$ and equality
\eqref{eq:g2-counterexample-color} are inherited by $s$.
If $a=b$, apply (G2) in $s$
with base $a$, noting that $\min\{u,v\}<y\leq a$.  Suppose now that
$a\ne b$, and let
\begin{displaymath}
 j=\max\{\ell\leq k:a_\ell=b_\ell\}.
\end{displaymath}
Then $j<k$.  For every $j\leq\ell<k$, both
$(a_\ell,a_{\ell+1})$ and $(b_\ell,b_{\ell+1})$ are old; otherwise
$(F_r)$ would make the two ordered pairs identical, contradicting the
choice of $j$.  It follows from the induction above that all vertices
in the two terminal sequences
\begin{displaymath}
 a_j,a_{j+1},\ldots,a_k=a
 \quad\text{and}\quad
 b_j,b_{j+1},\ldots,b_k=b
\end{displaymath}
belong to $A^s$.  Let
\begin{displaymath}
 z=\min\{a_\ell,b_\ell:j\leq\ell\leq k\}.
\end{displaymath}
If $z=a_j=b_j$, use the two terminal sequences displayed above.
Otherwise use
\begin{displaymath}
 z,a_j,a_{j+1},\ldots,a_k=a
 \quad\text{and}\quad
 z,b_j,b_{j+1},\ldots,b_k=b.
\end{displaymath}
The first steps of these two sequences are identical, and the remaining
corresponding steps have equal $f^s$-colors.  Thus in either case the
resulting sequences are synchronized walks in $A^s$ witnessing
\begin{displaymath}
 (a,b)\in\mathsf{Sync}^{z,0}_s.
\end{displaymath}
Since $z\geq y$ and $\min\{u,v\}<y$, (G2) for $s$ contradicts
\eqref{eq:g2-counterexample-color}.  This proves (G2), and hence $r$
is good.
\end{proof}

\begin{proof}[Proof of Theorem~\ref{thm:main}]
We verify that $P$ has the properties asserted in the theorem.
The empty precondition is good.  Lemma~\ref{lem:fresh} shows that, for
every $\alpha<\omega_1$, the set
\begin{equation}\label{eq:dense-support}
 \{p\in P:\alpha\in A^p\}
\end{equation}
is dense.  It also shows that, for every $n<\omega$ and $i<2$, it is
dense to have an $n$-element subset of $A^p$ all of whose pairs receive
the $h^p$-color $i$.

Clearly $|P|=\omega_1$.  We check the c.c.c.  Given an uncountable
family of conditions, first thin it so that its supports are distinct;
this is possible because only countably many conditions have any one
fixed finite support.  Apply the $\Delta$-system lemma to the supports.
After thinning, they have a common finite root $R$ and pairwise
disjoint, nonempty petals.  Only countably many petals can meet
$\max(R)+1$, so we may discard them and assume that $R$ lies below
every petal (with the evident interpretation when $R=\varnothing$).
There are only countably many possibilities for the sizes of the
supports and for the finite $h$- and $f$-matrices.  We may therefore
thin again so that every increasing bijection between two supports
fixes $R$ and is an isomorphism.  Fix one petal $D_p$.  Only countably
many of the other pairwise disjoint petals meet $\max(D_p)+1$, so one
may choose a petal $D_q$ with $D_p<D_q$.  The corresponding conditions
are separated twins, so
Lemma~\ref{lem:amalgamation} gives a common extension.  Thus $P$ is
c.c.c.

Let $K\subseteq P$ be generic.  By~\eqref{eq:dense-support}, the unions
\begin{displaymath}
 h=\bigcup_{p\in K}h^p:[\omega_1]^2\longrightarrow2,
 \qquad
 f=\bigcup_{p\in K}f^p:\omega_1^2\longrightarrow\omega
\end{displaymath}
are total.  Put $E_i=h^{-1}\{i\}$ and $H_i=(\omega_1,E_i)$.
Then $H_1=\overline{H_0}$.  Requirement (G1) implies that $f$ is a
proper coloring of $H_0\times H_1$, so
\begin{equation}\label{eq:product-upper-bound}
 \chr(H_0\times H_1)\leq\omega.
\end{equation}
The dense sets mentioned after \eqref{eq:dense-support} ensure that
each $H_i$ contains
$K_n$ for every $n<\omega$.  Pairing a copy of $K_n$ in $H_0$ with a
copy in $H_1$ embeds $K_n$ into $H_0\times H_1$.  Hence equality holds
in~\eqref{eq:product-upper-bound}.

We finish by proving that $\chr(H_i)=\omega_1$ for $i<2$.  Suppose
that some $p_*\in P$ forces that
$\dot c:\omega_1\to\omega$ is a proper coloring of $H_i$.  For each
$\alpha$ in the stationary set of nonzero countable limit ordinals,
choose $p_\alpha\leq p_*$ such that $\alpha\in A^{p_\alpha}$ and
$p_\alpha$ decides $\dot c(\alpha)$.  The map
\begin{displaymath}
 \alpha\longmapsto
 \begin{cases}
 0,&A^{p_\alpha}\cap\alpha=\varnothing,\\
 \max(A^{p_\alpha}\cap\alpha)+1,&A^{p_\alpha}\cap\alpha\ne\varnothing
 \end{cases}
\end{displaymath}
is regressive.  By the pressing-down lemma it is constant on a
stationary set.  On that set the finite sets
$A^{p_\alpha}\cap\alpha$ are all contained in one countable ordinal;
there are only countably many possibilities for them.  Thus, after
thinning, they are a fixed finite set $L$.

Apply the $\Delta$-system lemma to the supports and thin once more.
Its root must equal $L$: after discarding boundedly many indices, every
root element lies below $\alpha$, while every member of $L$ belongs to
every support.  Thinning by the countably many finite matrix types and
by the countably many values of $\dot c(\alpha)$, we may arrange that
\begin{enumerate}
\item the $A^{p_\alpha}$ have the common root $R=L$ and form pairwise
isomorphic conditions;
\item $R<A^{p_\alpha}\setminus R$, and
$\alpha=\min(A^{p_\alpha}\setminus R)$;
\item all the conditions force the same value, say $m$, for
$\dot c(\alpha)$.
\end{enumerate}
Choose $\alpha<\beta$ in the remaining uncountable set so that the two
petals are successive blocks.  Their increasing isomorphism sends
$\alpha$ to $\beta$.  Lemma~\ref{lem:amalgamation}, used with the
prescribed color $i$, gives $r\leq p_\alpha,p_\beta$ such that
$h^r(\{\alpha,\beta\})=i$.  But
\begin{displaymath}
 r\Vdash \dot c(\alpha)=\dot c(\beta)=m,
\end{displaymath}
contrary to the asserted properness of $\dot c$.  Thus neither $H_i$
is countably chromatic.  Since $P$ is c.c.c.\ and each $H_i$ has vertex
set $\omega_1$, we have
$\chr(H_0)=\chr(H_1)=\omega_1$.  Taking $G=H_0$ completes the proof.
\end{proof}

\section*{Declaration of Generative AI and AI-assisted technologies in the writing process}

During the preparation of this work, the author used ChatGPT to
correct grammatical and typographical errors and misprints.  All
mathematical arguments were created by the author.  After using this
tool, the author reviewed and edited the content as needed and takes
full responsibility for the content of the publication.


\begin{thebibliography}{99}

\bibitem{CummingsEisworthMoore2024}
J.~Cummings, T.~Eisworth, and J.~T. Moore,
\emph{On minimal non-$\sigma$-scattered linear orders},
Adv. Math. \textbf{441} (2024), article 109540,
doi:10.1016/j.aim.2024.109540.

\bibitem{Hajnal1985}
A.~Hajnal,
\emph{The chromatic number of the product of two
$\aleph_1$-chromatic graphs can be countable},
Combinatorica \textbf{5} (1985), no.~2, 137--139,
doi:10.1007/BF02579376.

\bibitem{Hedetniemi1966}
S.~T. Hedetniemi,
\emph{Homomorphisms of graphs and automata},
University of Michigan Technical Report 03105--44--T, 1966.

\bibitem{Moore2007}
J.~T. Moore,
\emph{$\omega_1$ and $-\omega_1$ may be the only minimal uncountable
linear orders},
Michigan Math. J. \textbf{55} (2007), no.~2, 437--457,
doi:10.1307/mmj/1187647002.

\bibitem{NordhausGaddum1956}
E.~A. Nordhaus and J.~W. Gaddum,
\emph{On complementary graphs},
Amer. Math. Monthly \textbf{63} (1956), 175--177,
doi:10.2307/2306658.

\bibitem{PoljakRodl1981}
S.~Poljak and V.~R\"odl,
\emph{On the arc-chromatic number of a digraph},
J. Combin. Theory Ser. B \textbf{31} (1981), 190--198,
doi:10.1016/S0095-8956(81)80024-X.

\bibitem{Rinot2017}
A.~Rinot,
\emph{Hedetniemi's conjecture for uncountable graphs},
J. Eur. Math. Soc. \textbf{19} (2017), no.~1, 285--298,
doi:10.4171/JEMS/666.

\bibitem{Sauer2001}
N.~Sauer,
\emph{Hedetniemi's conjecture---a survey},
Discrete Math. \textbf{229} (2001), 261--292,
doi:10.1016/S0012-365X(00)00213-2.

\bibitem{Shalev2024}
R.~Shalev,
\emph{More minimal non-$\sigma$-scattered linear orders},
Eur. J. Math. \textbf{10} (2024), article no.~74,
doi:10.1007/s40879-024-00780-y.

\bibitem{Shitov2019}
Y.~Shitov,
\emph{Counterexamples to Hedetniemi's conjecture},
Ann.\ of Math. (2) \textbf{190} (2019), no.~2, 663--667,
doi:10.4007/annals.2019.190.2.6.

\bibitem{Soukup1988}
L.~Soukup,
\emph{On chromatic number of product of graphs},
Comment. Math. Univ. Carolin. \textbf{29} (1988), no.~1, 1--12.

\bibitem{Zhu2025}
X.~Zhu,
\emph{A survey on Hedetniemi's conjecture},
Comput. Sci. Rev. \textbf{58} (2025), Article 100794,
doi:10.1016/j.cosrev.2025.100794.

\end{thebibliography}
\end{document}